\documentclass[conference]{IEEEtran}
\IEEEoverridecommandlockouts

\usepackage[utf8]{inputenc} 
\usepackage[T1]{fontenc}
\usepackage{url}
\usepackage{ifthen}
\usepackage{cite}
\usepackage[cmex10]{amsmath} 
\usepackage{xcolor}
\usepackage{graphicx}
\usepackage{subcaption}
\usepackage{comment}

\renewcommand{\vec}[1]{\mathbf{#1}}

\usepackage{algorithm}
\usepackage{algpseudocode}
\usepackage{amssymb}
\newcommand\norm[1]{\left|\left|#1\right|\right|}

\newcommand\alg[1]{\left\langle#1\right\rangle}

\newcommand\argmin{\operatorname{argmin}}

\newcommand{\be}{\begin{equation}}
\newcommand{\ee}{\end{equation}}
\newcommand{\ba}{\begin{aligned}}
\newcommand{\ea}{\end{aligned}}
\newcommand{\ben}{\begin{enumerate}}
\newcommand{\een}{\end{enumerate}}
\newcommand{\bit}{\begin{itemize}}
\newcommand{\eit}{\end{itemize}}

\newtheorem{theorem}{Theorem}[section]

\newtheorem{corollary}{Corollary}[section]
\newtheorem{lemma}{Lemma}[section]
\newtheorem{proposition}{Proposition}[section]
\newtheorem{definition}{Definition}[section]

\newtheorem{remark}{Remark}[section]

\makeatletter
\newcommand{\linebreakand}{%
  \end{@IEEEauthorhalign}
  \hfill\mbox{}\par
  \mbox{}\hfill\begin{@IEEEauthorhalign}
}
\makeatother

\newenvironment{proof}{\textit{Proof}:}{\hfill$\square$}

\begin{document}
\title{Aspects of a Generalized Theory of Sparsity based Inference in Linear Inverse Problems\thanks{This research is sponsored by the Office of Naval Research (ONR) Code 31, Grant 1401091801 (RGR), and Grants N00014-23-1-2394 and N00014-23-1-2790 (HNM).}}


\author{\IEEEauthorblockN{Raghu G. Raj\IEEEauthorrefmark{1}}
  \and
  \IEEEauthorblockN{Ryan O'Dowd\IEEEauthorrefmark{2}}
  \and
  \IEEEauthorblockN{\textcolor{black}{Owen T. Huber\IEEEauthorrefmark{3}}}
  \and
  \IEEEauthorblockN{Hrushikesh N. Mhaskar\IEEEauthorrefmark{4}}
\linebreakand
\IEEEauthorrefmark{1}U.S. Naval Research Laboratory, Radar Division, Washington D.C.\\
\IEEEauthorrefmark{2}MRI and Spectroscopy Section, National Institute on Aging, National Institutes of Health, Baltimore, MD\\
\IEEEauthorrefmark{3}Georgia Institute of Technology, Department of Mathematics, Atlanta, GA\\
\IEEEauthorrefmark{4}Institute of Mathematical Sciences, Claremont Graduate University, Claremont California
}

\maketitle


\begin{abstract}
Linear inverse problems are ubiquitous in various science and engineering disciplines. In recent decades, applying sparsity-based priors, such as $\ell_1$ priors, to linear inverse problems has led to the development of compressive sensing (CS) and sparsity-based signal processing. More recently, methods based on a Compound Gaussian (CG) prior have been investigated and demonstrate improved results over CS in practice. This paper is the first attempt to identify and elucidate the fundamental structures underlying the success of CG methods by studying CG in the context of a broader framework of generalized-sparsity-based-inference. After defining our notion of generalized sparsity, we introduce a weak null space property and proceed to generalize two well-known methods in CS, basis pursuit and iteratively reweighted least squares (IRLS). We show how a subset of CG-induced regularizers fits into this framework, and provide numerical simulations suggesting that examples of G-IRLS offer superior convergence rates in traditional sparsity problems and superior image reconstruction in tomographic imaging problems.
\end{abstract}

\section{introduction}
Linear inverse problems are ubiquitous in various science and engineering disciplines with applications in fields such as medical imaging (e.g., computer tomography (CT) and magnetic resonance imaging (MRI) reconstruction), geophysics (e.g., seismic tomography), astrophysics (e.g., image deconvolution), computational biology (e.g., gene expression analysis), signal processing (e.g., denoising and source separation), machine learning (e.g. feature selection in sparse regression models), and remote sensing (e.g., atmospheric data retrieval and image restoration). These problems take the form
\begin{equation}\label{eq:linearproblem}
	\mathbf{y}=A\mathbf{c}+\boldsymbol{\nu},
\end{equation}
where $\mathbf{y}\in \mathbb{R}^m$ is the observed data, $A\in \mathbb{R}^{m\times n}$ is the effective sensing matrix, $\boldsymbol{\nu}\in\mathbb{R}^m$ is the additive noise vector, and $\mathbf{c}\in\mathbb{R}^n$ is the vector of unknown coefficients to be estimated. \textcolor{black}{When $m\ll n$, we call the linear system underdeterminated (or overcomplete).}

The field of sparsity-based signal processing, and its sub-field of compressive sensing (CS), considers the problem of inverting \eqref{eq:linearproblem} under the assumption of a sparsity constraint on $\mathbf{c}$, \textcolor{black}{where sparsity refers to the $\ell_0$ (pseudo)-norm of the vector.} Numerous algorithms for solving this subclass of linear inverse problems have been studied over the past two decades including greedy \cite{CoSamp}, convex optimization \cite{l1ls}, and iterative thresholding \cite{ISTA}. The success of an algorithm in finding such a sparse solution \textcolor{black}{is guaranteed only when properties of the effective sensing matrix are satisfied}, such as the restricted isometry property \cite{candes-tao-rip} or null space property  \cite{IRLS, CohenDahmenDevore09}.

\textcolor{black}{A new class of algorithms and network structures for solving linear inverse problems based on Compound Gaussian (CG) induced regularizers\cite{CG-LS-Net, CG-DR-Net, HB-MAP}, which reduce to $\ell_1$-based algorithms under limiting conditions, have shown improvement over state-of-the-art algorithms when applied to tomographic imaging and CS problems.} Furthermore, a statistical learning theory for CG-based neural networks that theoretically confirms the numerical results obtained in tomographic imaging and CS applications has also been developed \cite{CG-IT1}.

Nevertheless, there is currently a lack of mathematical theory that gives a quantitative understanding of basic questions underlying CG inference, such as the precise nature of the generalization with respect to (w.r.t.) traditional sparsity based inference, convergence rates of CG algorithms etc.

This paper is a first attempt to identify and elucidate the fundamental structures underlying the success of CG methods by studying them in the context of a broader framework of generalized-sparsity-based-inference. We contribute and expand on existing literature for linear inverse problems and compressive sensing by demonstrating how this broadened notion of sparsity can be applied. In particular, we introduce a corresponding ``weak" null space property and use it to deduce new results, which generalize both basis pursuit and iteratively re-weighted least squares (IRLS). At the same time, we investigate how CG-based regularizers fit into this theory. The aim of this paper is construct the foundations towards a full theory for generalized-sparsity-based-inference with the CG prior.

Throughout this paper, we refer to vectors by lowercase bold-face letters. The notation $A^T$ denotes the transpose of the real matrix, $A$. The notation $\norm{\textbf{a}}_p$ denote the $\ell_p$-norm. The symbols $\mathbb{R}$, $\mathbb{R}^+$, and $\mathbb{N}$ denote, respectively, the real, positive real, and natural numbers. Finally, unless stated otherwise, all unbolded lowercase letters denote either scalars or sequences drawn from real or complex fields.

The rest of this paper is organized as follows. In Section~\ref{sec:background}, we cover relevant background information for this paper including an introduction to the CG-based inference methods (\ref{subsec:cg}). In Section~\ref{sec:gsparsity}, we introduce a new notion of generalized sparsity, and with it, define a new weak null space property. We conclude the section with a generalization of the theorem for basis pursuit. In Section~\ref{sec:cg-connection}, we explore how a subset of CG distributions generate regularizers that satisfy the requirements of the theory in this paper. In Section~\ref{sec:g-irls}, we introduce a generalized version of the IRLS algorithm and give conditions for when the algorithm will estimate regularized least squares solutions. In Section~\ref{sec:Num_Res} we present numerical results demonstrating the benefit of generalizing the IRLS algorithm. Lastly, in Section~\ref{sec:discussion}, we give closing remarks and point to some questions which can help guide the future work stemming from this paper.

\section{Background}
\label{sec:background}

In this paper, we examine the problem of recovering values of $\vec{c}$ from observations of the form $\vec{y}$ as in \eqref{eq:linearproblem}, where $A\in \mathbb{R}^{m\times n}$ is full rank, $\vec{c}\in\mathbb{R}^n$, $\vec{\nu}\in\mathbb{R}^m$, and $m\ll n$. We set 
\be
G_\vec{y}:=\{\vec{c}\in \mathbb{R}^n:A\vec{c}=\vec{y}\}.
\ee 
Define the set $[n]:=\{1,2,\dots,n\}$ and suppose $S\subseteq [n]$. For any $\vec{x}\in\mathbb{R}^n$, we introduce the notation $\vec{x}_S$ to be the vector with entries given by
\be
[\vec{x}_{S}]_i:=\begin{cases}x_i & \text{if }i\in S\\ 0 & \text{else}\end{cases}
\ee
for all $i\in [n]$.
Furthermore, we define subtraction of sets $X,Y\subseteq \mathbb{R}^n$ in the usual sense by
\be
X-Y=\{x-y:x\in X,y\in Y\}.
\ee
In Table~\ref{tab:dists}, we highlight some useful distributions \textcolor{black}{that are referenced} in this paper. We use the notation $V\overset d= W$ for random variables $V,W$ that are equal in distribution.

\begin{table}[!ht]
\begin{center}
\caption{Distributions used in this paper, including their name, symbol, and probability distribution functions.}
\label{tab:dists}
\footnotesize
\vspace{10pt}
\begin{tabular}{|p{15mm}|c|c|}
\hline
Name & Symbol & Pdf\\
\hline
Normal\newline distribution & $\mathcal{N}(\mu,\sigma^2)$ & $\frac{1}{\sqrt{2\pi}\sigma}e^{-(x-\mu)^2/2\sigma^2}$\\
\hline
Multivariate\newline Normal\newline distribution & $\mathcal{N}(\mu,\Sigma)$ & $\frac{1}{(2\pi)^{n/2}\Sigma^{1/2}}e^{-\vec{(x-\mu)}^T\Sigma^{-1}\vec{(x-\mu)}/2}$\\
\hline
Laplace\newline distribution & $\mathcal{L}(\mu,\lambda)$ & $\frac{\lambda}{2}e^{-\lambda|x-\mu|}$\\
\hline
Rayleigh\newline distribution & $\mathcal{Z}(\sigma^2)$ & $\begin{cases}\frac{x}{\sigma^2}e^{-x^2/2\sigma^2} &\text{if }x\geq 0\\ 0 &\text{else}\end{cases}$\\
\hline
\end{tabular}
\end{center}
\end{table}

\subsection{CG-based Inference}
\label{subsec:cg}

The CG approach to the problem \eqref{eq:linearproblem} can be understood through a Bayesian maximum a posteriori estimation with a prior on $\vec{c}$ of the form
\be
C\overset d=Z\odot U,
\ee
where $U\sim \mathcal{N}(0,\Sigma_U)$, and $Z$ is a mixing distribution supported on $[0,\infty)$. Given the existence of such a representation, we say the random variable, $C$ belongs to a \textit{CG distribution}. Then, with the assumption that $p_{Y|C}(\vec{y}|\vec{c})=\boldsymbol{\nu}\sim \mathcal{N}(A\vec{c}-\vec{y},\overline{\sigma} I)$, we \textcolor{black}{may define the loss function, $L$, from the maximum a posteriori (MAP) objective}
\be\label{eq:L}
\ba
L(\vec{c}):=&\underset{\vec{c}\in \mathbb{R}^n}{\operatorname{argmax}}\ p_{C|Y}(\vec{c}|\vec{y})\\
=&\underset{\vec{c}\in \mathbb{R}^n}{\operatorname{argmax}}\ p_{Y|C}(\vec{y}|\vec{c})p_{C}(\vec{c})\\
=&\underset{\vec{c}\in \mathbb{R}^n}{\operatorname{argmin}}-\log(p_{Y|C}(\vec{y}|\vec{c}))-\log(p_C(\vec{c}))\\
=&\underset{\vec{c}\in \mathbb{R}^n}{\operatorname{argmin}}\frac{1}{2\overline{\sigma}}\norm{A\vec{c}-\vec{y}}_2^2+R(\vec{c}),
\ea
\ee
where $R(\vec{c})=\log(p_C(0)/p_C(\vec{c}))$. We deduce from definitions that
\be\label{eq:cdist}\footnotesize
p_C(\vec{c})=\int_0^\infty \frac{1}{(2\pi)^{n/2}\Sigma_U^{1/2}}\exp\left(-\frac{\left(\vec{c}\odot \vec{z}^{-1}\right)^T\Sigma^{-1}_U\left(\vec{c}\odot \vec{z}^{-1}\right)}{2}\right)\frac{1}{\vec{z}}d\vec{z},
\ee
where $\vec{z}^{-1}=(1/z_1,1/z_2,\dots,1/z_n)$.
In the remainder of the paper we consider the case where $\Sigma_U=\operatorname{diag}(\sigma_1^2,\sigma_2^2,\dots,\sigma_n^2)$ in \eqref{eq:cdist}. Then $p_C(\vec{c})=\prod_{i\in [n]} p_{C_i}(c_i)$ where for each $i\in [n]$
\be\label{eq:CG_pdf}
p_{C_i}(c_i)=\int_0^\infty \frac{1}{\sqrt{2\pi}\sigma_i}\exp\left(-\frac{(c_i/z_i)^2}{2\sigma_i^2}\right)\frac{p_{Z_i}(z_i)}{z_i}dz_i.
\ee
Defining $R_i(c_i)=\log(p_{C_i}(0)/p_{C_i}(c_i))$, we write
\be\label{eq:R}
R(\vec{c})=\sum_{i\in[n]} R_i(c_i).
\ee
\begin{remark} 
When $p_C(\vec{c})=\prod_{i\in[n]} p_{C_i}(c_i)$ and each $C_i\sim \mathcal{L}(0,1)$ \textcolor{black}{we have} $R=\norm{\circ}_1$.
\end{remark}

\section{Weak null space property and generalized sparsity}
\label{sec:gsparsity}

In the remainder of the paper, we assume $R:\mathbb{R}^n\to \mathbb{R}$ is an even, subadditive function with $R(\vec{0})=\vec{0}$. First, we will state some properties of  $R$ that follow from these conditions. 

\begin{lemma}\label{lem:induction}
	If $R$ is zero at the origin, even, and subadditive, then for all integers $n > 0$,
	\begin{equation}
		R(nx) \leq nR(x).
	\end{equation}
\end{lemma}
\begin{proof}
	We will prove this by induction. The $n = 1$ case is trivial. For $n = 2$, we prove that $$R(2x) = R(x + x) \leq R(x) + R(x) \leq 2R(x).$$ Assume that $R(kx) \leq kR(x)$ for all $k \in \mathbb{N}$. Then,
	\begin{equation*}
		\begin{aligned}
			R((k + 1)x) &\leq R(x) + R(kx) \leq R(x) + kR(x),
		\end{aligned}
	\end{equation*}
 and hence $R((k + 1)x) \leq (k + 1)R(x)$.
\end{proof}

 \begin{lemma}
 	If $R$ is zero at the origin, even, and subadditive, then 
 	\begin{enumerate}
 		\item $R(\mathbf{x)} \geq 0$ for all $\mathbf{x} \in \mathbb{R}^n$.
 		\item $R$ cannot be strictly convex in a region containing the origin.
 		\item $R$ cannot be continuously differentiable on an interval including the origin. 
 		\item If $R(x) >0$ for some $x\in \mathbb{R}$, then, for all $\delta > 0$, there exists an $x$ such that $0 < x < \delta$ and $R(x) > 0$.
 	\end{enumerate}
 \end{lemma}
 \begin{proof}
 	
 	1) Using the subadditivity of $R$, we find that $$0 = R(0) = R(x + (-x)) \leq R(x) + R(-x) = 2R(x).$$
 	
 	2) Suppose, for sake of contradiction, that $R$ is strictly convex for all $x \in [-\delta, x^*]$ where $x^* > 0, \delta \geq 0$.
 	
 	By strict convexity, for all $t \in [0, 1]$, $$R(tx^*) < (1 - t)R(0) + tR(x^*) = tR(x^*).$$ It follows that $R(x^*/2) < R(x^*)/2$. 
 	
 	However, by Lemma \eqref{lem:induction}, $R(nx) \leq nR(x)$ for all $x$ and integers $n$. Choose $n = 2$ and $x = x^*/2$. Then, we have that $R(x^*) \leq 2 R(x^*/2)$ implying that $R(x^*/2) \geq R(x^*)/2,$ which is a contradiction.  
 	
 	3) Since $R$ is even, differentiability at the origin implies $R'(0) = 0$ \footnote{\textcolor{black}{This statement is offered as a lemma in the Appendix.}}. Next, we will prove that there  exists some sequences $\{\xi_n\}$ such that $\lim_{n \rightarrow \infty }\xi_n = 0$ while $\lim_{n \rightarrow \infty} R'(\xi_n) \neq 0$, which is a contradiction by the continuity of $R'$.
 	
 	Choose some $x^* > 0$ such that $R(x^*) > 0$. Then, by Lemma \ref{lem:induction}, for all integers $n > 0$, $R(x/n) \geq R(x)/n$. Then, it follows that $R(x^*/n) \geq R(x^*)/n$ for all $n$. By the mean value theorem, for all $n$, there exists some $0 < \xi_n < x^*/n$ such that $$R'(\xi_n) = \frac{R(x^*/n) - R(0)}{x^*/n} \geq \frac{R(x^*)}{x^*}.$$
 	Then, $\lim_{x \rightarrow 0}R'(0) \neq 0$ because for any $\epsilon > 0$ such that $\epsilon < R(x^*)/x^*$, there exist $\xi_n$ arbitrarily close to the origin where $|R'(\xi_n) - R'(0)| = R'(\xi_n) > \epsilon$. So, $R$ must not be continuously differentiable around zero, and we conclude our proof. 
 
 	4) We will prove this lemma by contradiction, supposing that $R$ takes a nonzero value on $\mathbb{R}$ and $R$ does not take a nonzero value arbitrarily close to the origin. Consider $\mathbf{X} = \{x \, | \, R(x) > 0, x>0\}$. $\mathbf{X}$ is not empty, as there exists $x$ for which $R$ takes nonzero values, and $\mathbf{X}$ is bounded below by a number greater than zero, so we can define $\tilde{x} = \inf \mathbf{X}$. Note that $\tilde{x} \neq 0$ because there exists no sequences $\{x_n\}$ that converge to a nonzero $x$ such that $R(x) = 0$. 
 	
 	If $R(\tilde{x}) = 0$, then for every $\epsilon > 0$, there exists some $x^*$ such that $\tilde{x} < x^* < \tilde{x} + \epsilon$ and $R(x^*) > 0$. Choosing a $x^*$ for $\epsilon = \tilde{x}/2$, it follows that $2R(x^*/2) = 0$ because $x^*/2 < \tilde{x}$. Then, $2R(x^*/2) = 0 < R(x^*)$ which is a contradiction of Lemma \ref{lem:induction}.
 	
 	If $R(\tilde{x}) > 0$, then we immediately notice that $2R(\tilde{x}/2) = 0$. A contradiction arises in $2R(\tilde{x}/2) < R(\tilde{x})$ concluding our proof.
 \end{proof}
 
\color{black}
Next, we introduce new definitions of  regularizer-specific, generalized sparse vectors and a weak null space property. We assume $\epsilon\geq 0$ and $B\subseteq \mathbb{R}^n$.

\begin{definition} We say that a vector $\vec{x}\in \mathbb{R}^n$ is $(K,R,\epsilon)$-\textit{sparse} if there exists $S\subseteq [n]$ with $|S|\leq K$ such that
\be\label{eq:gsparse}
R(\vec{x}_{[n]\setminus S})\leq \epsilon.
\ee
We denote the space of all such vectors by $\Sigma_{K,R,\epsilon}$. For a vector $\vec{x}\in\mathbb{R}^n$, We define the degree of approximation by $(K,R,\epsilon)$-sparse vectors by
\be
\sigma_{K,R,\epsilon}(\vec{x}):=\inf_{\vec{x}'\in \Sigma_{K,R,\epsilon}}R(\vec{x}-\vec{x}').
\ee
\end{definition}

We note the existence of alternative notions of generalized sparsity \cite{plan-sparse,junge-sparse}. The novelty of our definition is that it is regularizer dependent, allowing for use with a broader range of minimization problems.

\begin{definition} We say that the tuple $(B,R)$ satisfies the $(K,\gamma,\delta)$-\textit{weak null space property} if for any $S\subseteq [n]$ where $|S|\leq K$ and any $\vec{x}\in B- B$ we have
\be\label{eq:wnull}
R(\vec{x}_S)\leq \gamma R(\vec{x}_{[n]\setminus S})+\delta.
\ee
\end{definition}

We also introduce the set of near minimizers of $R$.

\begin{definition}
We say that a point $\vec{x}\in B$ is an $\epsilon$-\textit{near minimizer} of $R$ over $B$ if
\be
R(\vec{x})\leq \inf_{\vec{x}'\in B} R(\vec{x}')+\epsilon.
\ee
We denote the set of all such $\vec{x}$ values by $\mathcal{M}(R,\epsilon,B)$.
\end{definition}

We end this section with a theorem relating these definitions through a generalized compressive sensing lens. This theorem is an extension of the traditional theorem for the convergence of basis pursuit to a sparse solution \cite[Theorem 4.5]{Compressive_Sensing}.

\begin{theorem}
Let $B \subseteq \mathbb{R}^n$, $R:\mathbb{R}^n\to\mathbb{R}$ be a subadditive function, and $\epsilon\geq 0$. 
\begin{enumerate}
	\item If $(B,R)$ satisfies the $(K,1,\delta)$-null space property, then $B\cap \Sigma_{K,R,\epsilon}\subseteq \mathcal{M}(R,2\epsilon+\delta,B)$.
	
	\item Conversely, if for every $\vec{v}\in B-B$ and $S\subseteq [n]$ with $|S|\leq K$, $\vec{v}_S\in \mathcal{M}(R,\epsilon, B)$ and $-\vec{v}_{[n]\setminus S}\in B$, then $(B,R)$ satisfies the $(K,1,\epsilon)$-null space property.
\end{enumerate}
\end{theorem}

\begin{proof}
	
	1) \textcolor{black}{Choose some $\vec{x}\in B \cap \Sigma_{K,R,\epsilon}$, noting that $\vec{x}-\vec{x}'\in B-B$ for any $\vec{x}' \in B$. First, using the subadditivity of $R$ and the $(K,1,\delta)$-null space property of $(B, R)$, with some $S \subseteq [n]$ having $|S| \leq K$ whose existence is guaranteed from $\vec{x}$ being in $\Sigma_{K,R,\epsilon}$, we have that 
\begin{equation}
\begin{aligned}
R(\vec{x}_S)\leq& R((\vec{x}-\vec{x}')_S)+R(\vec{x}'_S)\\
\leq& R((\vec{x}-\vec{x}')_{[n]\setminus S})+R(\vec{x}'_S)+\delta.
\end{aligned}
\end{equation}
Using the subadditivity and evenness of $R$, 
\begin{equation}
	\begin{aligned}
		R(\vec{x}_S)\leq& R(\vec{x}_{[n]\setminus S})+R(\vec{x}'_{[n]\setminus S}) + R(\vec{x}'_S)+\delta\\
		 \leq& R(\vec{x}_{[n]\setminus S})+R(\vec{x}')+\delta.\\
		\leq& R(\vec{x}')+\delta+\epsilon.
	\end{aligned}
\end{equation}
Then, adding $R(\vec{x}_{[n]\setminus S})$ to both sides, and remembering that it is bounded above by $\epsilon$, we conclude that $R(\vec{x})\leq R(\vec{x}')+\delta+2\epsilon$.
}
Since $\vec{x}'\in B$ was arbitrary, this implies $\vec{x}\in \mathcal{M}(R,2\epsilon+\delta,B)$.
	
2) Let $\vec{v}\in B-B$ and $S\subseteq [n]$ with $|S|\leq K$. Given the assumptions of the theorem, we note
\be
R(\vec{v}_S)\leq \inf_{\vec{x}'\in B} R(\vec{x}')+\epsilon\leq R(-\vec{v}_{[n]\setminus S})+\epsilon=R(\vec{v}_{[n]\setminus S})+\epsilon.
\ee
\end{proof} 

\begin{remark}
One may notice that when $R=\norm{\circ}_1$, $\gamma=1$, $\delta,\epsilon=0$, and $B=G_\vec{y}$, then this generalized theorem reduces to \cite[Theorem 4.5]{Compressive_Sensing}.
\end{remark}

\section{Connection to the CG prior}
\label{sec:cg-connection}

In this section, we investigate a subset of CG priors, which we refer to as Compound Laplacian (CL), and outline some properties of their associated regularizers. We start with a proposition showing how the Laplacian distribution is a specific CG distribution with a Rayleigh distributed mixing variable. Since we are treating the random variables as independent in each vector component, we simply examine univariate distributions in the first two propositions of this section. In Proposition~\ref{prop:cg-subadditive}, we use these univariate results component-wise to demonstrate that regularizers from CL distributions exhibit the desired subadditive property.

\begin{proposition}\label{prop:cglaplace}
Let $\lambda,\sigma>0$ and
\be
\qquad Z\sim \mathcal{Z}(1/\sigma^2\lambda^2),\qquad U\sim \mathcal{N}(0,\sigma^2).
\ee
Then
\be
ZU\overset{d}{=}\mathcal{L}(0,\lambda).
\ee
\end{proposition}

\begin{proof}
In \cite{ding-laplace-gsm}, it was shown that a distribution of the form $\sqrt{E}U$, where $U$ is Gaussian and $E$ is an exponential random variable, is equivalent to the Laplacian distribution. By checking that $\sqrt{E}\overset d=Z$, we achieve the desired result.

\end{proof}

The following corollary is a direct consequence of Proposition~\ref{prop:cglaplace} and shows that CL distributions are a subset of CG distributions.

\begin{corollary}
	Any CL distribution defined as the multiplication of $L\sim \mathcal{L}(0,\lambda)$ and the random mixing variable $Z$ is equivalent to the CG distribution with the random mixing variable $Z^* = Z\overline{Z}$ for $\overline{Z}\sim \mathcal{Z}(1/\sigma^2\lambda^2)$ and the Gaussian $U\sim \mathcal{N}(0,\sigma^2)$.
\end{corollary}


We now give a proposition showing that any CL distribution yields a subadditive regularizer when defined through \eqref{eq:L}.
\begin{proposition}\label{prop:cg-subadditive}
Suppose $C=ZL$ is a CL distribution. Then $R$ as defined in \eqref{eq:R} is a subadditive, even function.
\end{proposition}

\begin{proof}
Since $R$ is defined component-wise, it will be subadditive if it is subadditive in each component. 
Let $i\in [n]$.
We can represent the pdf of $C_i$ by
\be\label{eq:fc}
p_{C_i}(c_i)=\int_0^\infty e^{-\lambda|c_i|/z_i}d\mu_i(z_i),
\ee
where $d\mu_i=\frac{\lambda}{2}\frac{p_{Z_i}(z_i)}{z_i}dz_i$, and note that $p_{C_i}$ is an even, non-increasing function.
From the definition $R_i(c_i)=\log(p_{C_i}(0)/p_{C_i}(c_i))$, it is clear that $R_i$ is non-decreasing on $[0,\infty)$ and an even function, so that $R_i(c_i)=R_i(|c_i|)$. 
A simple calculation shows that for $c_i\ge 0$,
\be\label{eq:R''}
R''_i(c_i)=\frac{p'_{C_i}(c_i)^2-p''_{C_i}(c_i)p_{C_i}(c_i)}{p_{C_i}(c_i)^2}.
\ee
Using Schwarz inequality and \eqref{eq:fc}, it is not difficult to verify that $R_i''\le 0$, so that $R_i$ is a concave function on $[0,\infty)$.
Since $R_i(0)=0$, this implies that $R_i$ is subadditive on $[0,\infty)$. 
Now, since $R_i$ is non-increasing on $[0,\infty)$ and even on $\mathbb{R}$, we see that for any $x,y\in\mathbb{R}$,
$$
\begin{aligned}
R_i(x+y)&=R_i(|x+y|)\le R_i(|x|+|y|)\\
&\le R_i(|x|)+R_i(|y|)=R_i(x)+R_i(y);
\end{aligned}
$$
i.e., $R_i$ is sub-additive on $\mathbb{R}$.
\end{proof}

\section{Generalized iteratively reweighted least-squares (G-IRLS)}
\label{sec:g-irls}

We introduce a generalized version of the IRLS loss function from \cite{IRLS}:
\begin{equation}\label{eq:g-irls-L}
L(\vec{c},\vec{w},\epsilon)=\frac{1}{2}\sum_{j=1}^n (c_j)^2w_j+\epsilon^2w_j+f_j(w_j),
\end{equation}
where each $f_j$ is an implicitly defined function of the regularizer, $R_j$. If the weights $w_j$ are updated to become 
\begin{equation}\label{eq:weightsset}
	w_j^k=R_j\left(\sqrt{(c_j^k)^2+\epsilon^2}\right)/((c_j^k)^2+\epsilon^2)
\end{equation}
 for each $j\in [n]$, the first terms of $L$ will approach $\frac{1}{2}R(\mathbf{c})$ as desired. Hence, the update step
\begin{equation}
	\vec{w}^{k+1}=\underset{\vec{w}\in (\mathbb{R}^n)^+}{\argmin}\ L(\vec{c}^{k+1},\vec{w},\epsilon_{k+1}),
\end{equation}
which results in the weights
\begin{equation}
	 \vec{w}_j^{k+1} \gets (f_j')^{-1}(-(\vec{c}^{k+1}_j)^2-\epsilon_k^2),
\end{equation}
should be set to the quantity in \eqref{eq:weightsset}, leading to the differential equation
\be\label{eq:f-cond}
f_j'(R_j(x)/x^2)=-x^2
\ee
that defines $f_j$ in terms of $R_j$. In practice, we simply update $w_j^{k + 1}$ as the RHS of \eqref{eq:weightsset}. In the traditional IRLS, with $R$ defined as the $L_1$ norm, \eqref{eq:weightsset} resolves to $f_j(x)=1/x$ for each $j\in [n]$.

Our formulation of $L$ requires some further restrictions on the class of regularizers we consider in this section. In addition to $R$ being subadditive, even, and $R(\vec{0})=\vec{0}$, we would also like $R$ to be continuous, differentiable and invertible on $(0,\infty)$, and chosen so there exist $f_j$'s satisfying \eqref{eq:f-cond}.

  With $\mathbf{c}$ updated as 
  \begin{equation}
  	\vec{c}^{k+1}=\underset{\vec{c}\in G_y}{\argmin}\ L(\vec{c},\vec{w}^k,\epsilon_k),
  \end{equation}
  we consider the generalized IRLS (G-IRLS) algorithm in Algorithm~\ref{alg:g-irls}. We denote the iteration number of an output vector by a superscript in terms of $k$ and denote the vector resulting from $R_j$ being applied to each component of a vector $\mathbf{x}$ as $\mathbf{R}(\mathbf{x})$. We also define the function $r(\mathbf{x})_K$ as that which returns \textit{the $K^{th}$ largest element of $|\mathbf{x}|$}.

\begin{algorithm}
\caption{Generalized iteratively reweighted least squares (G-IRLS)}\label{alg:irls}
\textbf{Input:} Full-rank $A\in\mathbb{R}^{m\times n}$, $\vec{y}\in \mathbb{R}^m$, $K\in [N]$, $\overline{k}\in\mathbb{N},\overline{\epsilon}\in \mathbb{R}^+$.\\
\textbf{Initialization:} $\vec{w}^0=\vec{1}\in\mathbb{R}^n$, $\epsilon_0=n$, $k=0$.
\begin{algorithmic}
\While{$\epsilon_k\geq \overline{\epsilon}$ and $k\leq \overline{k}$}
\State $D_k \gets \operatorname{diag}(\vec{w}^k)^{-1}$
\State $\vec{c}^{k+1} \gets D_kA^T(AD_kA^T)^{-1}\vec{y}$
\State $\epsilon_{k+1}\gets \min\left\{\epsilon_k,r(\mathbf{R}(\vec{c}^{k+1}))_{K+1}\right\}$
\For{$j\in [n]$}
\State $\vec{w}_j^{k+1} \gets (f_j')^{-1}(-(\vec{c}^{k+1}_j)^2-\epsilon_k^2)$
\EndFor
\State $k\gets k+1$\Comment{increment $k$ by one and repeat loop.}
\EndWhile
\end{algorithmic}
\textbf{Return:} $\overline{\vec{c}}=\vec{c}^{k}$.
\label{alg:g-irls}
\end{algorithm}


The results of this section are focused on showing how G-IRLS can be used to obtain solutions belonging to $\Sigma_{K,R,\epsilon}$. We also give bounds in terms of the degree of approximation by $(K,R,\epsilon)$-sparse vectors. We begin this section with some preparatory results.

\begin{proposition}\label{prop:bestapproxbound}
Let $\vec{x}\in \mathbb{R}^n$ and choose $S\subseteq [n]$ to be the indices of the $K$ largest components of $\mathbf{R}(\vec{x})$. If $R$ is strictly increasing on $(0, \infty)$, then
\be\label{eq:Rbest}
R(\vec{x}_{[n]\setminus S})\leq \sigma_{K,R,\epsilon}(\vec{x})+\epsilon.
\ee
\end{proposition}

\begin{proof}
	First we will show that $\Sigma_{K, R, \epsilon}$ is a closed set. Let $(\vec{x}^{k})_k$ be a convergent sequence with $\vec{x}^k\in \Sigma_{K,R,\epsilon}$ for all $k\in \mathbb{N}$ and with the limit $\vec{x}^*$. For sake of contradiction, suppose that $R(\vec{x}^*_{[n]\setminus S})> \epsilon$ for all $S\subseteq [n]$ with $|S|\leq K$. We equivalently suppose that 
	\begin{equation}\label{eq:equal_state}
		R(\vec{x}^*_{[n]\setminus S}) \geq \epsilon + \delta 
	\end{equation}
	for the $\delta > 0$ defined by $$\delta = \underset{S \subseteq [n] \text{ s.t. } |S| \leq K}{\inf}  R(\vec{x}_{[n] \setminus S})  - \epsilon.$$

 Since $R$ is continuous, $R(\vec{0})=\vec{0}$, and $R$ is invertible on $(0,\infty)$, there exists some $N$ such that when $k\geq N$, 
\be
\lVert \vec{x}^{k}-\vec{x}^* \rVert_{\infty}<R^{-1}\left(\frac{\delta}{2n}\right).
\ee
Since $\vec{x}^{k}\in\Sigma_{K,R,\epsilon}$, there exists $\bar{S}\subseteq[n]$ with $|\bar{S}|\leq K$ such that
\be
\ba
R(\vec{x}^*_{[n]\setminus \bar{S}})\leq&  R(\vec{x}^{k}_{[n]\setminus \bar{S}})+R(\vec{x}^{k}_{[n]\setminus \bar{S}}-\vec{x}^*_{[n]\setminus \bar{S}})\\
\leq& \epsilon+nR\left(R^{-1}\left(\frac{\delta}{2n}\right)\right)=\epsilon+\frac{\delta}{2}.
\ea
\ee
This is a contradiction of \eqref{eq:equal_state}, so there must exist some $S$ such that $R(\vec{x}^*_{[n]\setminus S})\leq \epsilon$, implying that $\Sigma_{K,R,\epsilon}$ is a closed set. Thus, there exists some $\overline{\vec{x}}\in \Sigma_{K,R,\epsilon}$ such that $\sigma_{K,R,\epsilon}(\vec{x})=R(\vec{x}-\overline{\vec{x}})$. Furthermore, when $\overline{\vec{x}}_S=\vec{x}_S$, using the subadditivity of $R$, we conclude that
\be
\ba
R(\vec{x}_{[n]\setminus S})=&R(\vec{x}_{[n]\setminus S}-\overline{\vec{x}}_{[n]\setminus S}+\overline{\vec{x}}_{[n]\setminus S})\\
\leq& R((\vec{x}-\overline{\vec{x}})_{[n]\setminus S})+R(\overline{\vec{x}}_{[n]\setminus S})\\
\leq& \sigma_{K,R,\epsilon}(\vec{x})+\epsilon.
\ea
\ee
\end{proof}

\begin{lemma}\label{lem:1}
Suppose that $(B,R)$ satisfies the $(K,\gamma,\delta)$-null space property with parameter $\gamma<1$. Then, for any $\vec{x},\vec{x}'\in B$,
\be\label{eq:lem1result1}
R(\vec{x}'-\vec{x})\leq\frac{\gamma+1}{1-\gamma}\left(R(\vec{x}')-R(\vec{x})+2\sigma_{K,R,\epsilon}(\vec{x})+2\epsilon+\delta\right).
\ee
Let $\vec{x}^*\in \Sigma_{K,R,\epsilon}$. Then $\vec{x}^*\in\mathcal{M}(R,2\epsilon+\delta,B)$ and for any $\vec{x}\in B$ we have
\be\label{eq:lem1result2}
R(\vec{x}-\vec{x}^*)\leq 2\frac{\gamma+1}{1-\gamma}(\sigma_{K,R,\epsilon}(\vec{x})+2\epsilon+\delta).
\ee
\end{lemma}

\begin{proof}
In this proof only, we denote the indices of the $K$ largest values of $\mathbf{R}(\vec{x})$ by $S$. We can see
\be\label{eq:lem1-1}
\ba
R((\vec{x}'&-\vec{x})_{[n]\setminus S})\leq R(\vec{x}'_{[n]\setminus S})+\sigma_{K,R,\epsilon}(\vec{x})+\epsilon\\
=&R(\vec{x}_S)-R(\vec{x}'_S)+R(\vec{x}')-R(\vec{x})+2\sigma_{K,R,\epsilon}(\vec{x})+2\epsilon\\
\leq&R((\vec{x}'-\vec{x})_S)+R(\vec{x}')-R(\vec{x})+2\sigma_{K,R,\epsilon}(\vec{x})+2\epsilon.
\ea
\ee
By the $(K,\gamma,\delta)$-null space property, Equation~\eqref{eq:lem1-1}, and the assumption $\gamma<1$ we can see
\be\label{eq:lem1-2}
R((\vec{x}'-\vec{x})_S)\leq \frac{\gamma \left(R(\vec{x}')-R(\vec{x})+2\sigma_{K,R,\epsilon}(\vec{x})+2\epsilon\right)+\delta}{1-\gamma}.
\ee
Combining \eqref{eq:lem1-1}~and~\eqref{eq:lem1-2}, we get \eqref{eq:lem1result1} since $(1+\gamma)\delta>\delta$. Supposing that $\vec{x}^*$ is $(K,R,\epsilon)$-sparse, then $\sigma_{K,R,\epsilon}(\vec{x}^*)=0$. By \eqref{eq:lem1result1} we can then see $R(\vec{x}^*)\leq R(\vec{x})+2\epsilon+\delta$ for any $\vec{x}\in B$, implying that $\vec{x}^*\in \mathcal{M}(R,2\epsilon+\delta,B)$. Then \eqref{eq:lem1result2} is just a consequence of \eqref{eq:lem1result1} in the context $\vec{x}'=\vec{x}^*$.
\end{proof}

Before stating our theorem, we note that $( \epsilon_k )_k$ is a monotonically decreasing sequence bounded below by $0$, so it is convergent. We define $\epsilon:=\lim_{k\to \infty}\epsilon_k$ and assume that $\epsilon>0$, because otherwise it would imply the existence of an exactly $K$-sparse solution and one would have no need to use G-IRLS over IRLS. Essential to our theorem and proof are the functions
\be\label{eq:h}
h_{j,\epsilon}(x):=\frac{1}{2}\left[R_j\left(\sqrt{x^2+\epsilon^2}\right)+f_j\left(\frac{R_j\left(\sqrt{x^2+\epsilon^2}\right)}{x^2+\epsilon^2}\right)\right].
\ee
We note that $h_{j,\epsilon}$ is a differentiable function since $\epsilon>0$ and $R$ is differentiable on $(0,\infty)$ by assumption.
\begin{theorem}\label{thm:g-irls}
Suppose $\vec{y}\in \mathbb{R}^m$, $K\in [n]$, and $\gamma<1$ are given such that $(G_\vec{y},R)$ satisfies the $(K,\gamma,\delta)$-null space property. 
\begin{enumerate}
	\item If $h_{j,\epsilon}$ is a strictly convex function for each $j\in [n]$, then the generalized IRLS algorithm converges to a vector $\overline{\vec{c}}\in G_y\cap \Sigma_{K,R,\epsilon}\cap \mathcal{M}(R,2\epsilon+\delta,G_\vec{y})$. Furthermore, for any $\kappa \leq K$ and any $\vec{x}\in G_y$ we have
	\be\label{eq:g-irls1}
	R(\vec{x}-\overline{\vec{c}})\leq 2\frac{1+\gamma}{1-\gamma}(\sigma_{\kappa ,R,\epsilon}(\vec{x})+2\epsilon+\delta).
	\ee
	
	\item If $G_y\cap \Sigma_{\kappa,R,\epsilon}\neq \emptyset$ for some $0<\kappa<K-\frac{4+6\gamma}{1-\gamma}$, then 
		\begin{equation}\label{eq:g-irls2}
			\epsilon\leq\frac{2(1+\gamma)\delta}{(1-\gamma)(K-\kappa)-4-6\gamma}.
		\end{equation}
\end{enumerate}
\end{theorem}

\begin{proof}
Part 1) 
Since $\epsilon>0$ by assumption, there exists some $C>0$ such that $w_j^k\geq C$ for all $j\in[n]$, $k\in \mathbb{N}$. As a consequence of the same argument as in \cite[Lemma 5.1]{IRLS}, we have
\be
\sum_{k=1}^\infty \norm{\vec{c}^k-\vec{c}^{k+1}}_2^2\leq \lim_{k\to\infty} 2L(\vec{c}^k,\vec{w}^k,\epsilon_k)/C\leq 2H/C,
\ee
which implies that $\lim_{k\to \infty} \vec{c}^k-\vec{c}^{k+1}=0$.
we know that the sequence $(\vec{c}^k)_k$ has a convergent subsequence, since it is bounded. Let $(\vec{c}^{k_i})_i$ be such a subsequence, with limit $\vec{c}^*$. The sequence $(w_j^{k_i})_i$ converges for each $j$ and, in particular,
\be
w^*_j:=\lim_{i\to\infty} w_j^{k_i}=\frac{R_j\left(\sqrt{(c_j^*)^2+\epsilon^2}\right)}{(c_j^*)^2+\epsilon^2}.
\ee
Let $\vec{p}\in G_\vec{0}$. We note \cite[Sec. 2]{IRLS} that for each $i$, $\alg{\vec{c}^{k_i},\vec{p}}_{\vec{w}^{k_i}}=0$, so $\alg{\vec{c}^*,\vec{p}}_{\vec{w}^*}=0$ in the limit as $i\to \infty$.
If we define
\be
g_\epsilon(\vec{c}):=\sum_{j=1}^n h_{j,\epsilon}(c_j),
\ee
then we can observe from the strict convexity and differentiability of each $h_{j,\epsilon}$ that
\be\label{eq:g-convex}
g_\epsilon(\vec{v})> g_\epsilon(\vec{c}^*)+\sum_{j=1}^n h'_{j,\epsilon}(c_j^*)(v_j-c_j^*),
\ee
for any $\vec{v}\in G_\vec{y}$. Using \eqref{eq:f-cond} and the fact that $\vec{v}-\vec{c^*}\in G_\vec{0}$, one can deduce that 
\be
\sum_{j=1}^n h'_{j,\epsilon}(c_j^*)(v_j-c_j^*)=\alg{\vec{c}^*,\vec{v}-\vec{c}^*}_{\vec{w}^*}=0,
\ee
which, with \eqref{eq:g-convex}, implies that $\vec{c}^*$ is the unique minimizer of $g_\epsilon$. Since the limit of any convergent subsequence of $(\vec{c}^k)_k$ is $\vec{c}^*$, we know that the sequence itself is convergent and $\overline{\vec{c}}=\vec{c}^*$.

If $\epsilon_k=\epsilon$ for some $k$, then $\vec{c}^{j}\in \Sigma_{K,R,\epsilon}$ for all $j\geq k$ and thus, $\overline{\vec{c}}\in \Sigma_{K,R,\epsilon}$. Otherwise, $\epsilon_k>\epsilon$ for all $k$, in which case there exists a strictly decreasing subsequence of $(\epsilon_k)_k$, which we index by $(\epsilon_{k_i})_i$. We note then that $\epsilon_{k_i-1}> r(R(\vec{c}^{k_i}))_{K+1}$ by construction, so
\be
r(R(\overline{\vec{c}}))_{K+1}=\lim_{i\to \infty} r(R(\vec{c}^{k_i}))_{K+1} \leq\epsilon,
\ee
since by \cite[Lemma~4.1]{IRLS} $r$ is Lipschitz with constant $1$ in $\ell_\infty$ norm.
Thus, $\overline{\vec{c}}\in \Sigma_{K,R,\epsilon}$. Lemma~\ref{lem:1} gives us \eqref{eq:g-irls1} and the fact that $\overline{\vec{c}}\in \mathcal{M}(R,2\epsilon,G_\vec{y})$.

Part 2) Suppose that there exists some $\vec{x}\in G_\vec{y}\cap \Sigma_{\kappa,R,\epsilon}$. By \eqref{eq:g-irls1} we have
\be
\ba
(K+&1-\kappa)\epsilon\leq(K+1-\kappa)r(R(\overline{\vec{c}}))_{K+1}\\
&\le\sigma_{\kappa,R,\epsilon}(\overline{\vec{c}})+\epsilon\leq R(\overline{\vec{c}}-\vec{x})+\sigma_{\kappa,R,\epsilon}(\vec{x})+\epsilon\\
&\le\frac{1+\gamma}{1-\gamma}(3\sigma_{\kappa,R,\epsilon}(\vec{x})+5\epsilon+2\delta)=\frac{1+\gamma}{1-\gamma}(5\epsilon+2\delta).
\ea
\ee
Then, rearrangement with the assumption that $K-\kappa>\frac{4+6\gamma}{1-\gamma}=5\frac{1+\gamma}{1-\gamma}-1$ completes the proof.
\end{proof}

\begin{remark}
Our theorem gives conditions for when $\overline{\vec{c}}$ is a near-minimizer of $R$ over $G_\vec{y}$, and bounds on $\epsilon$ based on the existence of generalized-sparse vectors in the solution space. One may note that G-IRLS reduces to IRLS when $\overline{\epsilon}=0$ and $R=\norm{\circ}_1$ (which implies $f_j(x)=1/x$ for each $j\in [n]$). In this case, when $\epsilon,\delta=0$, though $h_{j,0}$ are not strictly convex, one can show that the conclusions of our theorem reduce to statements equivalent to \cite[Theorem~5.3(i),(iv)]{IRLS}.
\end{remark}

\section{Numerical Results}
\label{sec:Num_Res}
In this section, we present numerical results that detail the performance of G-IRIS w.r.t. traditional sparsity based approaches including IRLS. We study the convergence rates of the algorithm in \ref{subsec:convrates} followed by performance in tomographic imaging in \ref{subsec:tomorecov}

\subsection{Convergence rates}
\label{subsec:convrates}

An advantage of the additional regularization terms available in G-IRLS is the accelerated rate of convergence for traditional sparsity problems. In \cite{IRLS}, where the regularizer acts as an $\ell_{\tau}$ penalty, it is shown that for $\tau = 1$, the algorithm is guaranteed to converge to the solution at a linear rate (if a K-sparse solution exists), while for $0 < \tau < 1$, though there is no guarantee of convergence, the iterations often converge at a superlinear rate. It is shown that once iterations become sufficiently close to the solution, convergence accelerates rapidly, which the authors suggest is due to the iterations entering a neighborhood of the solution where convergence is guaranteed for the concave regularization term.

It can be shown that every concave function $\tilde{R} : [0, \infty) \rightarrow \mathbb{R}$ such that $R(0) = 0$ and $R(x) \geq 0$ has the even extension $R(x) = \tilde{R}(|x|)$ that is subadditive, so we may use any regularization functions of this form in our G-IRLS framework. As done in \cite{IRLS}, it is advantageous to use $R(x)  = |x|$ for the first few iterations then a strictly concave regularization term afterwards for accelerated convergence.

\begin{figure*}[t]
	\centering
	\includegraphics[width=1.5\columnwidth]{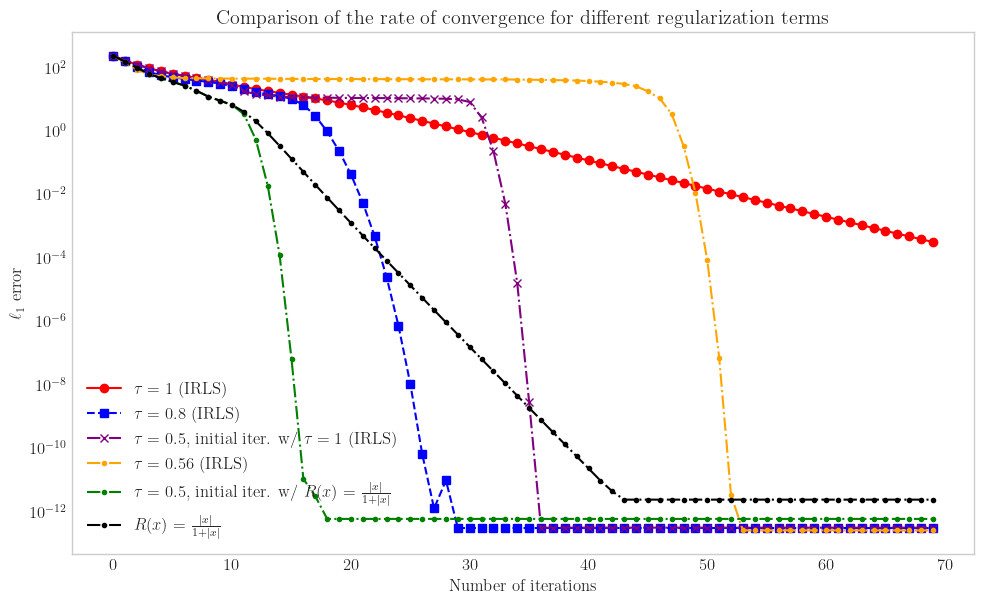}
	\caption{We show the decay of the logarithmic error for four examples of IRLS ($\tau = 1, \tau = 0.8, \tau = 0.5$ where the first ten iteration use $\tau = 1$, and $\tau = 0.56$) and two examples of G-IRLS, the first of which we perform G-IRLS with $R(x) = |x|/(1+|x|)$ for the first ten iterations then IRLS with $\tau = 0.5$ for all following iterations, and the second of which we use $R(x) = |x|/(1+|x|)$ for all iterations. }
	\label{fig:conv_rates}
\end{figure*}

 We use a sensing matrix $\Phi$ of size $250 \times 1500$ with Gaussian $\mathcal{N}(0, 1/250)$ i.i.d. entries and a ground truth of a 45-sparse vector in which the nonzero entries are sampled from a Gaussian. Results in Fig. \ref{fig:conv_rates} demonstrate that a mixture of using G-IRLS with $R(x) = |x|/(1 + |x|)$ outperforms all standard $\ell_{\tau}$ regularization terms for the standard sparse-recovery problem from \cite{IRLS}. We also note that the $R(x) = |x|/(1 + |x|)$ appears to show linear convergence that has a superior convergence rate to the $\tau = 1$ IRLS case. Finding other concave regularization terms that optimize the rate of convergence for the sparsity problem is a topic of future research. 

\subsection{Enhanced recovery in the tomographic setting}
\label{subsec:tomorecov}

It has been shown that the CG distribution with the log-normal distributed mixing variable, $Z \sim \text{Lognormal}(\mu, \sigma^2)$ having the pdf $f_Z(z) = \exp(-(\ln x - \mu)^2/2\sigma^2)/(x\sigma\sqrt{2\pi})$, effectively describes the statistics of wavelet coefficients for natural images \cite{CG-LS-Net, CG-IT1}. Accordingly, we use \eqref{eq:CG_pdf} to find that, when $\sigma = 1, \mu = 0$, the CG distribution with mixing variable $Z$ and an i.i.d. Gaussian with unit variance will have the pdf 
\begin{equation*}\label{eq:pdf_c}
	p_{C_i}(c_i) = \int_0^{\infty} \frac{1}{2\pi z_i^2 \sigma_i}\exp\left( -\frac{(c_i/z_i)^2}{2\sigma_i^2} -\frac{(\ln(z))^2}{2} \right) dz_i.
\end{equation*}
Using \eqref{eq:R}, we define the regularization term obtained from this distribution as 
\begin{equation}\label{eq:CG_R}
	R_{CG}(x) = \log\left( \frac{p_{C_i}(0)}{p_{C_i}(x)} \right),
\end{equation}
which can be calculated via numerical integration. 

In our experiments, we use a Radon transform with a number of uniformly spaced angles, $\Psi \in \mathbb{R}^{m \times n}$, as a sensing matrix, and refer to $\Phi \in \mathbb{R}^{n \times n}$ as our wavelet dictionary. Given the measurement $\mathbf{y} = \Psi \mathbf{c}$, where $\mathbf{c}$ is an image reshaped into a one-dimensional vector, we aim to recover $\mathbf{c}$ by solving $\mathbf{y} = \Psi \Phi \mathbf{x}$ for $\mathbf{x}$, wherein an approximation for $\mathbf{c}$ may be found from premultiplying $\mathbf{x}$ by the dictionary, $\Phi$. 

\begin{figure}[t]
	\centering
	\includegraphics[width=\columnwidth]{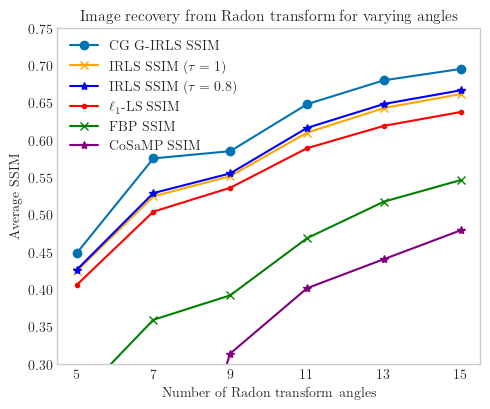}
	\caption{For the recovery problem involving sensing matrices of Radon transforms, with numbers of angles varying from 5 to 15, multiplied by a wavelet dictionary, we plot the average SSIM returned by G-IRLS with a CG regularizer and IRLS with $\tau = 1$ and $\tau = 0.8$. }
	\label{fig:different_angles}
\end{figure}

\begin{figure}[t]
	\centering
	
	\begin{subfigure}{0.32\columnwidth}
		\centering
		\includegraphics[width=\linewidth]{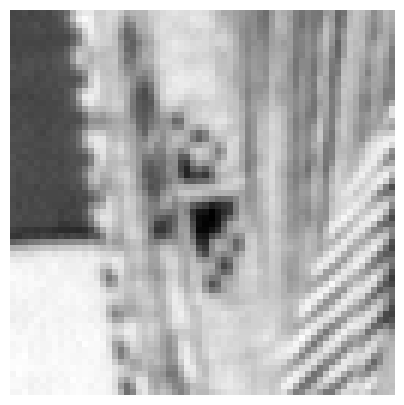}
		\caption{Original}
	\end{subfigure}
	\hfill
	\begin{subfigure}{0.32\columnwidth}
		\centering
		\includegraphics[width=\linewidth]{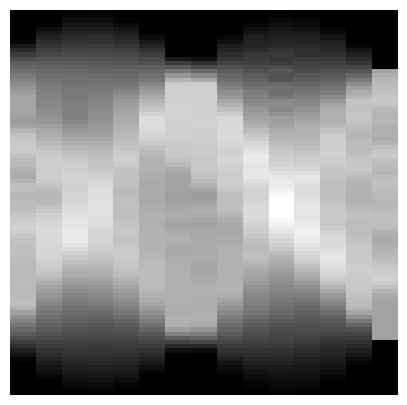}
		\caption{Radon Transform}
	\end{subfigure}
	\hfill
	\begin{subfigure}{0.32\columnwidth}
		\centering
		\includegraphics[width=\linewidth]{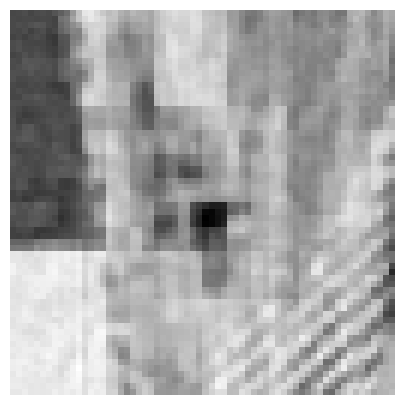}
		\caption{G-IRLS (0.761)}
	\end{subfigure}
	
	\vspace{2mm}
	
	\begin{subfigure}{0.32\columnwidth}
		\centering
		\includegraphics[width=\linewidth]{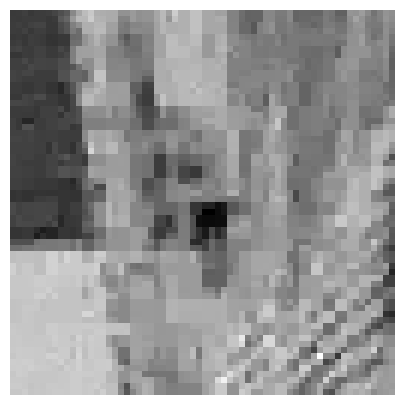}
		\caption{IRLS$_{\tau = 1}$ (0.719)}
	\end{subfigure}
	\hfill
	\begin{subfigure}{0.32\columnwidth}
		\centering
		\includegraphics[width=\linewidth]{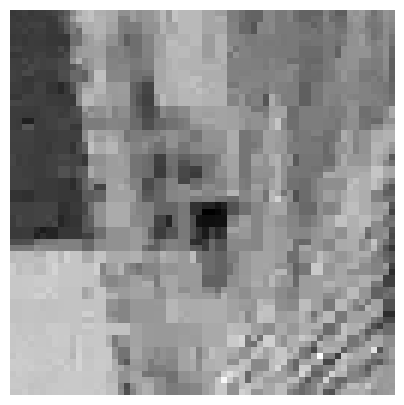}
		\caption{IRLS$_{\tau = 0.8}$ (0.710)}
	\end{subfigure}
	\hfill
	\begin{subfigure}{0.32\columnwidth}
		\centering
		\includegraphics[width=\linewidth]{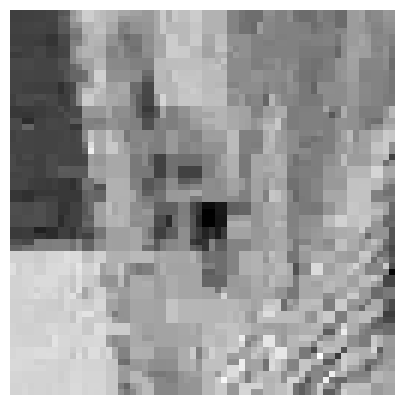}
		\caption{$\ell_1$-LS (0.688)}
	\end{subfigure}
	
		\vspace{2mm}
	
	\begin{subfigure}{0.32\columnwidth}
		\centering
		\includegraphics[width=\linewidth]{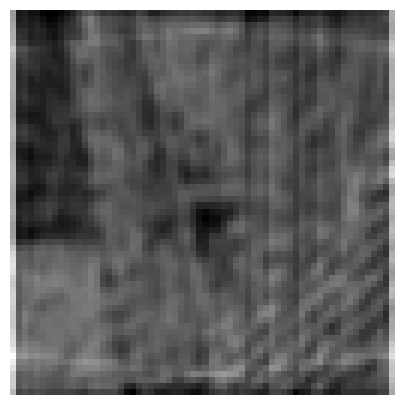}
		\caption{FBP (0.589)}
	\end{subfigure}
	\hspace{0.03\columnwidth}
	\begin{subfigure}{0.32\columnwidth}
		\centering
		\includegraphics[width=\linewidth]{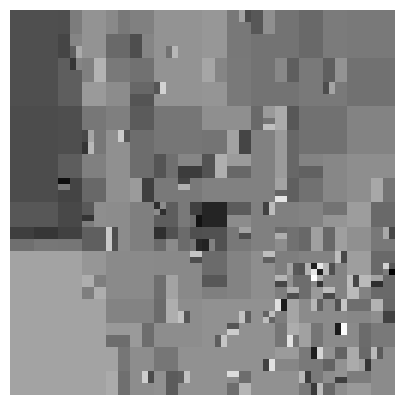}
		\caption{CoSaMP (0.470)}
	\end{subfigure}
	\hfill
	
	\caption{Image reconstructions (SSIM) using our G-IRLS with a CG based regularizer, IRLS, $\ell_1$-LS, FBP, and CoSaMP. The input to each algorithm is a vectorized Radon transform at 15 uniformly spaced angles of a 64$\times$64 image.}
	\label{fig:grid}
\end{figure}

We fix the sparsity parameter as $K = 184$ and compare the output SSIM of G-IRLS and IRLS for a Radon transform sensing matrix with varying number of angles, where fewer angles corresponds to a fatter sensing matrix. For the number of angles in $(5, 7, 9, 11, 13, 15)$ we measure the average SSIM over 15 randomly selected 64$\times$64 patches of the Barbara image and plot the average SSIM returned by G-IRLS with the regularizer from \eqref{eq:CG_R}, IRLS with $\tau = 1$ and $\tau = 0.8$, $\ell_1$-least squares ($\ell_1$-LS)~\cite{l1ls}, Fourier backprojection (FBP)~\cite{radontransform}, BCS~\cite{BCS}, and CoSaMP~\cite{CoSamp} in Fig. \ref{fig:different_angles}. For every angle, the plot shows that G-IRLS returns an image with superior SSIM on average. We also plot examples of the reconstructions by the different methods in Fig. \ref{fig:grid}.

\begin{figure}[t]
	\centering
	\includegraphics[width=\columnwidth]{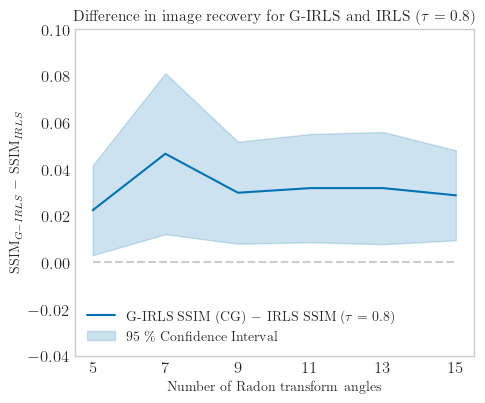}
	\caption{For the recovery problem in Fig. \ref{fig:different_angles}, we plot the average difference between SSIM returned by G-IRLS with a CG regularizer and IRLS with $\tau = 0.8$ with the 95\% confidence interval. }
	\label{fig:angles_diff}
\end{figure}

Fig. \ref{fig:angles_diff} plots the average difference in SSIM for images returned by G-IRLS and IRLS (with $\tau = 0.8$) along with 95\% confidence interval of this quantity. From the 95\% confidence interval lying above the zero line, it can be seen that G-IRLS with the CG based regularizer offers a statistically significant improvement over IRLS for tomographic image reconstruction. 

Additionally, Fig. \ref{fig:angles_CoSaMP} plots the average difference in SSIM for images returned by G-IRLS, IRLS (with $\tau = 1$), $\ell_1-LS$ \cite{l1ls}, Filtered Backprojection \cite{radontransform}, and CoSAMP \cite{CoSamp} along with 95\% confidence interval of the error quantities. From the 95\% confidence interval lying above the zero line, it can be seen that G-IRLS with the CG based regularizer offers a statistically significant improvement over all these algorithms for tomographic image reconstruction.

\begin{figure*}[t]
	\centering
	
	\begin{subfigure}{0.45\linewidth}
		\centering
		\includegraphics[width=\linewidth]{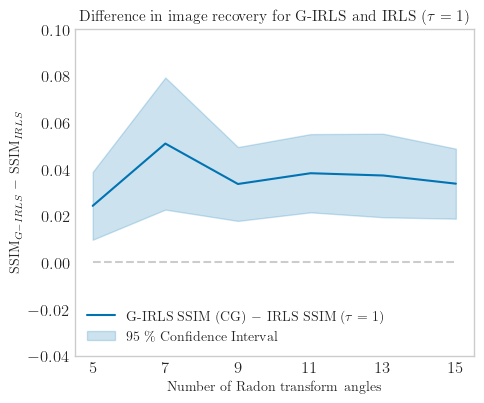}
		\caption{}
	\end{subfigure}
	\hfill
	\begin{subfigure}{0.45\linewidth}
		\centering
		\includegraphics[width=\linewidth]{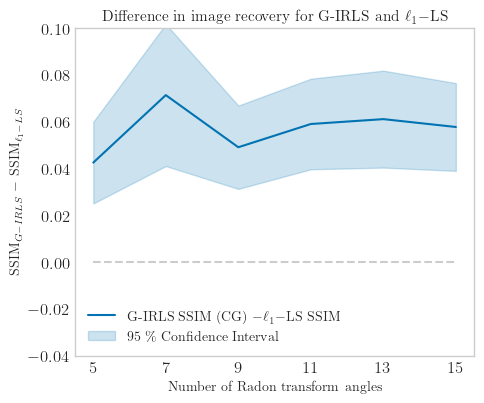}
		\caption{}
	\end{subfigure}	
	\vspace{2mm}
	
	\begin{subfigure}{0.45\linewidth}
		\centering
		\includegraphics[width=\linewidth]{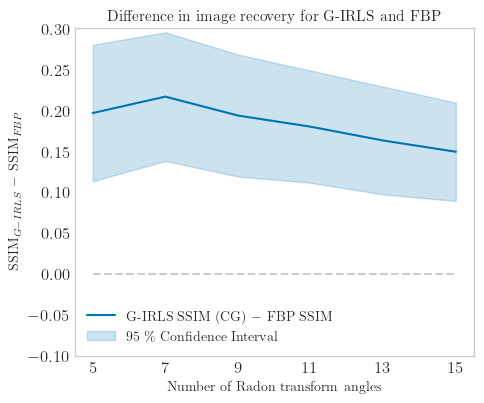}
		\caption{}
	\end{subfigure}
	\hfill
	\begin{subfigure}{0.45\linewidth}
		\centering
		\includegraphics[width=\linewidth]{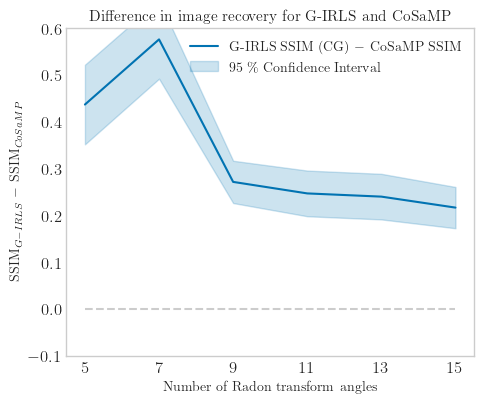}
		\caption{}
	\end{subfigure}	
	\caption{For the recovery problem in Fig. \ref{fig:different_angles}, we plot the average difference between SSIM returned by G-IRLS with a CG regularizer and IRLS with $\tau = 1$ (a), $\ell_1-$LS (b), FBP (c), and CoSaMP (d) with the 95\% confidence interval. }
	\label{fig:angles_CoSaMP}
\end{figure*}

\color{black}
\section{Discussion}
\label{sec:discussion}

In this paper, we have defined a new generalization for the notion of sparsity and a new weak null space property, which accommodates a wide range of regularizers, rather than just $\ell_1$, for the solution of linear inverse problems. In particular, we demonstrate that a subclass of CG distributions yield regularizers, through the maximum a posteriori estimate \eqref{eq:L}, that satisfy this theory. We have given novel theorems pertaining to generalizations of basis pursuit and the IRLS algorithm--both of which reduce to their original counterparts with the assumption that $R=\norm{\circ}_1$, among other simplifications. We offer numerical simulations demonstrate that regularizers beyond the $\ell_{\tau}$ norms can offer superior convergence rates in traditional sparsity problems and superior image reconstruction in tomographic imaging problems.



This paper opens the door to a new framework to view linear inverse problems where a new generalized notion of sparsity is a fundamental component. Although we have established some key connections among the theory of compressive sensing and CG methods, we recognize that there are further open questions that remain. Some questions of interest are:
\begin{enumerate}
\item How can the theory in this paper be extended to consider CG distributions more generally?

\item Need we restrict to the space of $G_\vec{y}$ in Theorem~\ref{thm:g-irls}? Our definition for the weak null space property may allow us to consider more broad sets of solutions.

\item How can the theory be extended when the parameters of the underlying CG distribution are known to lie on a lower dimensional manifold?

\item What is the role of the differential equation \eqref{eq:f-cond} in deciding whether a regularizer is well-suited for our algorithm?
\end{enumerate}

Our hope is that this paper inspires further investigation in the above directions, leading to a more complete theory for CG-based inference in a generalized sparsity setting.


\bibliographystyle{IEEEtran}
\bibliography{refs}

@book{radontransform,
  author={Deans, Stanley R},
  title={The {R}adon {T}ransform and {S}ome of {I}ts {A}pplications},
  year={2007},
  publisher = {Dover}
}

@article{ding-laplace-gsm,
author = {Peng Ding and Joseph K. Blitzstein},
title = {On the Gaussian Mixture Representation of the Laplace Distribution},
journal = {The American Statistician},
volume = {72},
number = {2},
pages = {172--174},
year = {2018},
publisher = {ASA Website}
}

@book{Compressive_Sensing,
author = {Foucart,Simon and Rauhut,Holger},
title = {A Mathematical Introduction to Compressive Sensing},
publisher = {Springer New York},
ISBN = {978-0-8176-4947-0},
year = {2013}
}

@article{ISTA,
  author={Daubechies, Ingrid and Defrise, Michel and De Mol, Christine},
  journal = {Commun. Pure Appl. Math},
  volume={57},
  number={11},
  title={An {I}terative {T}hresholding {A}lgorithm for {L}inear {I}nverse {P}roblems with a {S}parsity {C}onstraint},
  pages={1413--1457},
  year={2004},
  month = {Nov},
  publisher={Wiley Online Library},
  doi = {10.1002/cpa.20042}
}

@article{IRLS,
  author={Ingrid Daubechies and Ronald DeVore and Massimo Fornasier and C. Sinan Gunturk},
  journal = {Commun. Pure Appl. Math},
  volume={63},
  number={1},
  title={Iteratively re-weighted least squares minimization for sparse recovery}, 
  pages={1-38},
  year={2010},
  month = {Jan},
  publisher={Wiley Online Library},
  doi = {10.1002/cpa.20042}
}

@article{CG-LS-Net,
  author={Lyons, Carter and Raj, Raghu G. and Cheney, Margaret},
  journal={IEEE Transactions on Signal Processing}, 
  title={A Compound Gaussian Least Squares Algorithm and Unrolled Network for Linear Inverse Problems}, 
  year={2023},
  volume={71},
  number={},
  pages={4303-4316},
  doi={10.1109/TSP.2023.3330267}}

@article{HB-MAP,
    author = {Raj, Raghu G},
    journal = {Inverse Problems},
    volume = {32},
    number = {7},
    title = {A hierarchical {B}ayesian-{MAP} approach to inverse problems in imaging},
    ISSN = {0266-5611},
    pages = {075003},
    year = {2016},
    month = {Jul},
    publisher = {{IOP} Publishing},
    doi = {10.1088/0266-5611/32/7/075003}
}

@article{CG-IT1,
      title={On Generalization Bounds for Deep Compound Gaussian Neural Networks}, 
      author={Carter Lyons and Raghu G. Raj and Margaret Cheney},
      journal = {Submitted to IEEE Transactions on Information Theory},
      year={2024},
      eprint={2402.13106},
      archivePrefix={arXiv},
      primaryClass={stat.ML},
      url={https://arxiv.org/abs/2402.13106}
}

@ARTICLE{candes-tao-rip,
  author={Candes, E.J. and Tao, T.},
  journal={IEEE Transactions on Information Theory}, 
  title={Decoding by linear programming}, 
  year={2005},
  volume={51},
  number={12},
  pages={4203-4215},
  doi={10.1109/TIT.2005.858979}}

@ARTICLE{CG-DR-Net,
  author={Lyons, Carter and Raj, Raghu G. and Cheney, Margaret},
  journal={IEEE Transactions on Computational Imaging}, 
  title={Deep Regularized Compound Gaussian Network for Solving Linear Inverse Problems}, 
  year={2024},
  volume={10},
  number={},
  pages={399-414},
  doi={10.1109/TCI.2024.3369394}}

@article{plan-sparse,
author={Yaniv Plan and Roman Vershynin},
journal={Communications on Pure and Applied Mathematics},
title={One-Bit Compressed Sensing by Linear Programming},
year={2013},
Volume={66},
pages={1275-1297}
}

@article{junge-sparse,
    author = {Junge, Marius and Lee, Kiryung},
    title = {Generalized notions of sparsity and restricted isometry property. Part I: a unified framework},
    journal = {Information and Inference: A Journal of the IMA},
    volume = {9},
    number = {1},
    pages = {157-193},
    year = {2019},
    month = {02},
    issn = {2049-8772},
    doi = {10.1093/imaiai/iay018},
    url = {https://doi.org/10.1093/imaiai/iay018},
    eprint = {https://academic.oup.com/imaiai/article-pdf/9/1/157/32931220/iay018.pdf},
}

@article{BCS,
  author={Ji, Shihao and Xue, Ya and Carin, Lawrence},
  journal={IEEE Transactions on Signal Processing},
  volume={56},
  number={6},
  title={Bayesian {C}ompressive {S}ensing},
  ISSN = {1053-587X},
  pages={2346--2356},
  year={2008},
  month = {Jun}
}

@article{l1ls,
  author={Kim, Seung-Jean and Koh, Kwangmoo and Lustig, Michael and Boyd, Stephen and Gorinevsky, Dimitry},
  journal={IEEE Journal of Selected Topics in Signal Processing},
  volume={1},
  number={4},
  title={An {I}nterior-{P}oint {M}ethod for {L}arge-{S}cale $\ell_1$-{R}egularized {L}east {S}quares},
  ISSN = {1932-4553},
  pages={606--617},
  year={2007},
  month = {Dec}
}

@article{CoSamp,
  author={Needell, Deanna and Tropp, Joel A.},
  journal={Applied and Computational Harmonic Analysis},
  volume={26},
  number = {3},
  title={Co{S}a{MP}: Iterative signal recovery from incomplete and inaccurate samples},
  ISSN = {10635203},
  pages={301--321},
  year={2009},
  month = {May},
  publisher = {Elsevier}
}

@article{CohenDahmenDevore09,
author = {Cohen, Albert and Dahmen, Wolfgang and DeVore, Ronald},
year = {2009},
month = {01},
pages = {211-231},
title = {Compressed sensing and best k-term approximation},
volume = {22},
journal = {American Mathematical Society}
}

\appendix

\begin{lemma}
	If $R:\mathbb{R} \rightarrow \mathbb{R}$ is continuous, even, and $R(0) = 0$, then $R'(0) = 0$ if $R$ is differentiable at the origin.
\end{lemma}

\begin{proof}
	By the standard definition, $R'(0) = \lim_{h \rightarrow 0}(R(h) - R(0))/h$, and replacing $h$ with $-h$, we also have $R'(0) = \lim_{h \rightarrow 0}(R(-h) - R(0))/-h$ which is also equal to $\lim_{h \rightarrow 0}(R(h) - R(0))/-h = -R'(0)$ from the first equality. If $R'(0) = -R'(0)$ then $R'(0) = 0$.
\end{proof} \\

\begin{lemma}\label{lem:concave_nondecrease}
	If $R:[0, \infty) \rightarrow \mathbb{R}$ is concave, $R(0) = 0$, and $R(x)\geq 0$ for all $x$, then $R$ is nondecreasing. In this proof we assume that $R$ is twice continuously differentiable.  
\end{lemma}

\begin{proof}
	For sake of contradiction, assume that $R$ is decreasing at the point $x^*$ so that $R'(x^*) = m < 0$. Since $R$ is concave, $R''(x) \leq 0$ for all $x$ and $R'(x) \leq m$ for all $x> x^*$. With this, we can show that 
	\begin{align*}
		R(x) &= R(x^*) + \int_{x^*}^{x}R'(t) dt\\
		&\leq R(x^*) + \int_{x^*}^{x}m dt\\
		&\leq R(x^*) + m{x} - mx^*\\
	\end{align*}
	Choosing $\hat{x} = (mx^* - R(x^*) - 1)/m$ (which must be positive since $m$ is negative and $R$ takes positive values), we have that $R(\hat{x}) \leq -1 < 0$ which is a contradiction.  
\end{proof} \\

\begin{lemma}\label{lem:extend}
	If a function $R:[0, \infty) \rightarrow \mathbb{R}$ is subadditive, zero at the origin, and nondecreasing, then its even extension to the real numbers $G(x) = R(|x|)$ is also subadditive.  
\end{lemma}

\begin{proof}
	We want to show that $G(x) = R(|x|)$ is subadditive. That is, we want to show that $R(|x + y|) \leq R(|x|) + R(|y|)$. Since $R$ is subadditive on the nonnegative real numbers, we have that $R(|x| + |y|) \leq R(|x|) + R(|y|)$, and by the triangle inequality, $|x + y| \leq |x| + |y|$. Then, because $R$ is nondecreasing, we can conclude that $R(|x + y|) \leq R(|x| + |y|) \leq R(|x|) + R(|y|)$. 
\end{proof} \\

\begin{theorem}\label{thm:allconcok}
	Every twice differentiable concave function $R:[0, \infty) \rightarrow \mathbb{R}$ such that $R(0) = 0$ and $R(x) \geq 0$ has the even extension $G:\mathbb{R} \rightarrow \mathbb{R}$ defined by $G(x) = R(|x|)$ that is subadditive.
\end{theorem}

\begin{proof}
	From Lemmas \ref{lem:concave_nondecrease}, \ref{lem:extend}, we know that if we can prove that concave functions $R:[0, \infty) \rightarrow \mathbb{R}$ are subadditive, then its even extension is also subadditive. 
	
	From concavity, $$R((1 - t)x + ty) \geq (1 - t)R(x) + tR(y)$$for all $t \in [0, 1]$. Choosing $x = 0$, since $R(0) = 0$, we get the identity 
	\begin{equation}\label{eq:concineq}
		R(tx) \geq tR(x) 
	\end{equation}
	for all $t \in [0, 1]$. Next, for any $x, y \geq 0$, let us choose $t = x/(x + y)$, which gives us $1 - t = y/(x + y)$. These are both equivalent to the equations $x = t(x+y)$ and $y = (1 - t)(x+y) $. Using \eqref{eq:concineq}, we can find
	\begin{align*}
		R(x) &= R(t(x + y)) \geq tR(x+y) = \frac{x}{x + y}R(x + y)\\
		R(y) &= R((1-t)(x + y)) \geq (1-t)R(x+y) = \frac{y}{x + y}R(x + y)\\
	\end{align*}
	when we add these two expression, we find that $$R(x) + R(y) \geq \frac{x}{x + y}R(x + y) + \frac{y}{x + y}R(x + y) = R(x + y) $$and we conclude that $R$ must be subadditive. 
\end{proof}

\end{document}